\documentclass[reqno,a4paper,12pt]{amsart}

\usepackage[utf8]{inputenc}
\usepackage[T1]{fontenc}
\usepackage[english]{babel}
\usepackage[babel]{microtype}
\usepackage[a4paper, margin=1.1in]{geometry}
\usepackage[dvipsnames]{xcolor}
\usepackage{amsmath,amssymb,amsthm}
\usepackage{mathrsfs}
\usepackage{mathtools}
\usepackage{commath}
\usepackage{thmtools}
\usepackage{fixcmex}
\usepackage[pdfborderstyle={/S/U/W 0}]{hyperref}
\usepackage{cleveref}
\usepackage{etoolbox}
\usepackage{amsrefs}

\hypersetup{
    colorlinks=true,
    linkcolor=blue!85!black,
    citecolor=blue!85!black,
    urlcolor=blue!85!black,
}

\declaretheorem[style=plain,numbered=no,name=Theorem]{theorem}

\newcommand{\indef}[1]{\emph{#1}}

\renewcommand{\leq}{\leqslant}

\renewcommand{\geq}{\geqslant}

\begin{document}

\title{A note on the number of Egyptian fractions}

\author[N. Lebowitz-Lockard]{Noah Lebowitz-Lockard}
\address{Department of Mathematics, University of Texas, Tyler, TX 75799}
\email{nlebowitzlockard@uttyler.edu}

\author[V. Souza]{Victor Souza}
\address{Department of Pure Mathematics and Mathematical Statistics, University of Cambridge, Cambridge, United Kingdom}
\email{vss28@cam.ac.uk}

\begin{abstract}
Refining an estimate of Croot, Dobbs, Friedlander, Hetzel and Pappalardi, we show that for all $k \geq 2$, the number of integers $1 \leq a \leq n$ such that the equation $a/n = 1/m_1 + \dotsc + 1/m_k$ has a solution in positive integers $m_1, \dotsc, m_k$ is bounded above by $n^{1 - 1/2^{k-2} + o(1)}$ as $n$ goes to infinity.
The proof is elementary.
\end{abstract}

\maketitle


For a positive integer $k$, let $A_k(n)$ be the number of integers $1 \leq a \leq n$ for which $a / n$ has a $k$-term Egyptian fraction representation.
Namely, $A_k (n)$ is the number of $1 \leq a \leq n$ for which there are integers $m_1 \leq \dotsb \leq m_k$ such that
\begin{equation}
\label{eq:egyptian}
    \frac{a}{n} = \frac{1}{m_1} + \dotsb + \frac{1}{m_k}.
\end{equation}

Croot, Dobbs, Friedlander, Hetzel and Pappalardi~\cite{Croot2000-pk} proved that for any $k \geq 2$,
\begin{equation}
\label{eq:croot-bound}
    A_k(n) \leq n^{\alpha_k + o(1)}
\end{equation}
as $n$ goes to infinity, where $\alpha_k = 1 - 2/(3^{k-2} + 1)$.
In particular, this shows that $A_2(n) = n^{o(1)}$.
Refining their ideas, we get the following improved estimate.

\begin{theorem}
For every $k \geq 2$, we have $A_k(n) \leq n^{\beta_k + o(1)}$, where $\beta_k = 1 - 1/2^{k-2}$.
\end{theorem}
\begin{proof}
We perform induction on $k$.
The case $k = 2$ follows from \eqref{eq:croot-bound}, so assume $k \geq 3$.

Set $\gamma_i = 2^{i - 1}/2^{k}$.
For $1 \leq j \leq k -2 $, we say that such a tuple $(m_1, \dotsc, m_k)$ with $m_1 \leq \dotsb \leq m_k$ is of \indef{type $j$} if $m_j \geq n^{\gamma_j}$ and $m_i < n^{\gamma_i}$ for $1 \leq i < j$.
In addition, we say that such a tuple is of \indef{type $k-1$} if $m_i < n^{\gamma_i}$ for all $1 \leq i \leq k - 2$, no restrictions being placed on $m_{k-1}$ or $m_k$.
It follows from this definitions that any tuple $(m_1, \dotsc, m_k)$ has a type $1 \leq j \leq k - 1$.

We now show that the number of solutions to \eqref{eq:egyptian} of any given type is $\leq n^{\beta_k + o(1)}$, and the Theorem follows.
For instance, type 1 solutions satisfy
\begin{equation*}
    \frac{a}{n} = \frac{1}{m_1} + \dotsb + \frac{1}{m_k}
    \leq \frac{k}{m_1}
    \leq kn^{-\gamma_1}.
\end{equation*}
Thus, $a \leq kn^{1 - \gamma_1} = kn^{\beta_k}$.
That is, there are $\leq kn^{\beta_k}$ solutions of type 1.

Indeed, for type $j$ solutions, $1 \leq j \leq k - 2$, we have
\begin{equation*}
    0 \leq \frac{a}{n} - \frac{1}{m_1} - \dotsb - \frac{1}{m_{j-1}} = \frac{1}{m_j} + \dotsb + \frac{1}{m_k}
    \leq \frac{k-j+1}{m_j}
    \leq k n^{-\gamma_j}.
\end{equation*}
Therefore,
\begin{equation*}
    0 \leq a - \frac{n}{m_1} - \dotsb - \frac{n}{m_{j-1}}
    \leq k n^{1-\gamma_j},
\end{equation*}
or in other words, given $(m_1, \dotsc, m_{j-1})$, there are at most $k n^{1 - \gamma_j}$ choices for $a$.
Thus, the number of solutions of \eqref{eq:egyptian} of type $j$ is at most
\begin{equation*}
    kn^{1 - \gamma_j} m_1 \dotsb m_{j-1}
    \leq k n^{1 + \gamma_1 + \dotsb + \gamma_{j-1} - \gamma_j}
    = k n^{\beta_k}.
\end{equation*}

Solutions of type $k-1$ are handled more efficiently via binary Egyptian fractions.
Note that
\begin{equation*}
    0 \leq \frac{a}{n} - \frac{1}{m_1} - \dotsb - \frac{1}{m_{k-2}} = \frac{1}{m_{k-1}} + \frac{1}{m_k}.
\end{equation*}
So writing $a'= a m_1 \dotsb m_{k-2} - n m_1 \dotsb m_{k-2}(1/m_1 + \dotsb + 1/m_{k-2})$, we have that $a'$ is an integer satisfying $0 \leq a' \leq n m_1 \dotsb m_{k-2}$ and
\begin{equation*}
    \frac{a'}{n m_1 \dotsb m_{k-2}}
    = \frac{1}{m_{k-1}} + \frac{1}{m_k}.
\end{equation*}
Therefore, given $(m_1, \dotsc, m_{k-2})$, there are at most $A_2(n m_1 \dotsb m_{k-2})$ choices of $a$.
Thus, the number of choices of $a$ is bounded by
\begin{equation*}
    \sum_{\mathclap{\substack{m_i \leq n^{\gamma_i} \\ 1 \leq i \leq k-2}}} A_2(n m_1 \dotsb m_{k-2})
    \leq n^{\gamma_1 + \dotsb + \gamma_{k-2} + o(1)}
    = n^{\beta_k + o(1)},
\end{equation*}
as $A_2(n) = n^{o(1)}$ from \eqref{eq:croot-bound} and $n m_1 \dotsb m_{k-2} = n^{O(1)}$.
\end{proof}

We get an improved exponent for all $k \geq 4$.
To illustrate the change, the first exponents in \eqref{eq:croot-bound} are $\alpha_2 = 0$, $\alpha_3 = 1/2$, $\alpha_4 = 4/5$, $\alpha_5 = 13/14$ and $\alpha_6 = 79/81$, whereas the new exponents are $\beta_2 = 0$, $\beta_3 = 1/2$, $\beta_4 = 3/4$, $\beta_5 = 7/8$ and $\beta_6 = 15/16$.

It is still expected, however, that $A_k(n) = n^{o(1)}$ for all $k \geq 2$.
Nonetheless, even the weaker statement that $\sum_{n \leq x} A_k(n) = x^{1 + o(1)}$ remains unproven for $k \geq 3$.

Using our argument, if one shows that $A_k(n) = n^{o(1)}$ holds for some fixed $k$, we get that $A_{k + \ell}(n) \leq n^{1 - 1/2^\ell + o(1)}$ for all $\ell \geq 1$.
In fact, any improvement on the exponent in $A_k(n) \leq n^{\beta_k + o(1)}$ can be propagated to an improvement on $A_{k+\ell}(n)$.


\section*{Acknowledgements}

The authors are grateful for stimulating discussions with Carlo Adajar, Adva Mond and Julien Portier.
The second named author would like to thank his supervisor Professor Béla Bollobás for his continuous support.



\end{document}